\documentclass[twoside,12pt]{article}
\usepackage{amsthm,amsmath,amssymb,amscd,enumerate,epsfig}
\usepackage{amsfonts}
\usepackage{pst-all,fancybox}
\usepackage{graphicx}
\usepackage{fancyhdr}
\begin{document}
\textheight=570pt 
\begin{center} {\Large\bf On 2-Movable Total Domination in the Join and Corona of Graphs}
\vskip 1cm {\bf  Ariel C. Pedrano}\\
\
\\Mathematics and Statistics Department\\
College of Arts and Sciences\\
University of Southeastern Philippines\\
Davao City, Philippines\\
\vskip 1cm {\bf  Rolando N. Paluga}\\
\
\\Mathematics Department\\
College of Mathematics and Natural Sciences\\
Caraga State University\\
Butuan City, Philippines\\
\
\\email: ariel.pedrano@usep.edu.ph, rnpaluga@carsu.edu.ph\\
\end{center}

\begin{abstract}
Let $G$ be a connected graph. A non-empty $T\subseteq V(G)$ is a $2$-\textit{movable total dominating set} of $G$ if $T$ is a total dominating set and for every pair $x,y \in T$, $T \backslash \{x, y\}$ is a total dominating set in $G$, or there exist $u, v \in V(G) \backslash T$ such that $u$ and $v$ are adjacent to $x$ and $y$, respectively, and $(T \backslash \{x,y\}) \cup \{u,v\}$ is a total dominating set in $G$. The $2$-\textit{movable total domination number} of $G$, denoted by $\gamma_{mt}^{2}(G)$, is the minimum cardinality of a 2-movable total dominating set of $G$. A 2-movable total dominating set with cardinality equal to $\gamma_{mt}^{2}(G)$ is called \textit{$\gamma_{mt}^{2}$-set} of $G$. This paper present the 2-movable total domination in the join and corona of graphs.
\end{abstract}
\begingroup
\newtheorem{theorem}{Theorem}[section]
\newtheorem{lemma}[theorem]{Lemma}
\newtheorem{proposition}[theorem]{Proposition}
\newtheorem{corollary}[theorem]{Corollary}
\newtheorem{definition}[theorem]{Definition}
\newtheorem{remark}[theorem]{Remark}
\endgroup
\def\E#1{\left\langle #1 \right\rangle}
\def\diag{\mathop{\fam 0\relax diag}\nolimits}

\section{Introduction}

All graphs considered in this paper are all connected, finite, simple and undirected. Let $G=(V, E)$ be a connected, finite, simple and undirected graph. The graph $G$ has a vertex set $V=V(G)$ and an edge set $E=E(G)$. Further, let the order of the graph $G$ be $p$, that is, $|V|=|V(G)|=p$ and the size of the graph $G$ be $q$, that is, $|E|=|E(G)|=q$.

One of the interesting fields of Graph Theory is graph domination. Several variants of domination has been introduced and explored such as total domination \cite{3a}, 1-movable domination in graphs \cite{7}, neighborhoud transversal domination of some graphs \cite{b} and others. Inspired by the study of Blair et. al \cite{2} about 1-movable domination, Pedrano and Paluga \cite{p} defined a new variant of domination and called it 2-movable domination. A non-empty $S\subseteq V(G)$ is a $2$-\textit{movable dominating set} of $G$ if $S$ is a dominating set and for every pair $x,y \in S$, $S \backslash \{x, y\}$ is a dominating set in $G$, or there exist $u, v \in V(G) \backslash S$ such that $u$ and $v$ are adjacent to $x$ and $y$, respectively, and $(S \backslash \{x,y\}) \cup \{u,v\}$ is a dominating set in $G$. The $2$-\textit{movable domination number} of $G$, denoted by $\gamma_{m}^{2}(G)$, is the minimum cardinality of a 2-movable dominating set of $G$. A 2-movable dominating set with cardinality equal to $\gamma_{m}^{2}(G)$ is called \textit{$\gamma_{m}^{2}$-set} of $G$. 

In this study, the concept of 2-movable total domination in a graph will be introduced
and initially investigated. In particular, the 2-movable total domination number will be given for the join and corona of graphs.

\section{Basic Concepts}

\begin{definition}
	\cite{7} A subset $S$ of $V(G)$ is a dominating set of $G$ if for every $v \in V(G) \backslash S$, there exists $u \in S$ such that $uv \in E(G)$, that is, $N_G[S]=V(G)$. The domination number of $G$ is denoted by $\gamma(G)$ which refers to the smallest cardinality of a dominating set of $G$. A dominating set of $G$ with cardinality equal to $\gamma(G)$ is called a $\gamma-$set of $G$.
\end{definition}

\begin{definition}
	\cite{3a} Let $G$ be a graph without isolated vertices. A subset $S$ of $V(G)$ is a total dominating set of $G$ if $N_G(S)=V(G)$. That is, every vertex in $V(G)$ is adjacent to some vertex in $S$. 
\end{definition}

\begin{definition}\label{Defn4}
	\cite{p} Let $G$ be a connected graph. A non-empty $T\subseteq V(G)$ is a $2$-\textit{movable total dominating set} of $G$ if $T$ is a total dominating set and for every pair $x,y \in T$, $T \backslash \{x, y\}$ is a total dominating set in $G$, or there exist $u, v \in V(G) \backslash T$ such that $u$ and $v$ are adjacent to $x$ and $y$, respectively, and $(T \backslash \{x,y\}) \cup \{u,v\}$ is a total dominating set in $G$. The $2$-\textit{movable total domination number} of $G$, denoted by $\gamma_{mt}^{2}(G)$, is the minimum cardinality of a 2-movable total dominating set of $G$. A 2-movable total dominating set with cardinality equal to $\gamma_{mt}^{2}(G)$ is called \textit{$\gamma_{mt}^{2}$-set} of $G$.
\end{definition}

\section{Main Results}

\begin{theorem}
	If $S$ is a 2-movable total dominating set of $G$, then $S$ is a 2-movable dominating set of $G$. Furthermore, $\gamma_{m}^2(G) \leq \gamma_{mt}^2(G)$.
\end{theorem}\vspace{0.2mm}

\begin{proof}
	Suppose $S$ is a 2-movable total dominating set of $G$. It follows that $S$ is a total dominating set of $G$ by Definition \ref{Defn4}. Thus, $S$ is a dominating set of $G$. Since $S$ is a 2-movable total dominating set of $G$, then for every $a_1, a_2 \in S$, there exists $b_1, b_2 \in V(G) \backslash S$ such that $a_1b_1, a_2b_2 \in E(G)$ and $(S \backslash \{a_1, a_2\}) \cup \{b_1, b_2\}$ is a total dominating set of $G$. It follows that $(S \backslash \{a_1, a_2\}) \cup \{b_1, b_2\}$ is a dominating set of $G$. Hence, $S$ is 2-movable dominating set of $G$. Therefore, $$\gamma_{m}^2(G) \leq \gamma_{mt}^2(G).$$
\end{proof}

The next remark follows directly from Definition \ref{Defn4}:
\begin{remark}\label{Rmk:12}
	For any connected graph $G$ of order $n \geq 4$, $\gamma_{mt}^2(G) \geq 2$.
\end{remark}


\begin{theorem}\label{Thm:14}\rm
	Let $G$ and $H$ be graphs of order at least 2. Then
	\begin{center}
		$\gamma_{mt}^2 (G+H)=2$.
	\end{center}
\end{theorem}

\begin{proof}
	Let $G$ and $H$ be graphs of order at least 2, where $V(G+H) = V(G) \cup V(H)$ and $|V(G+H)| \geq 4$. Let $S=\{u,v\}$ where $u \in V(G)$ and $v \in V(H)$. We want to show that $S$ is a total dominating set of $G+H$. Let $x \in V(G+H)$. If $x \in V(G)$, then $xv \in E(G+H)$. If $x \in V(H)$, then $xu \in E(G+H)$. Therefore, $S$ is a total dominating set of $G+H$.
	
	Now, since $|V(G)| \geq 2$ and $|V(H)| \geq 2$, there exists $u' \in V(G)$ and $v' \in V(H)$ such that $u \not= u'$ and $v \not= v'$. Thus, $S'=(S \backslash \{u,v\}) \cup \{u',v'\} = \{u',v'\}$. By following the same arguments above, $S'$ is a total dominating set of $G+H$. Hence, $S$ is a 2-movable total dominating set of $G+H$.
	
	Therefore, $\gamma_{mt}^2 (G+H) \leq 2$. By Remark \ref{Rmk:12},  $\gamma_{mt}^2 (G+H) = 2$.
\end{proof}

\begin{theorem}\label{Thm:15}\rm
	Let $G$ be a graph of order at least 3. Then
	\begin{center}
		$\gamma_{mt}^2 (G+K_1)=\gamma_t (G)$.
	\end{center}
\end{theorem}

\begin{proof}
	Let $S$ be a $\gamma_t$-set of $G$. Let $x \in V(G+K_1)$. Suppose $x \in V(G)$. Since $S$ is a total dominating set of $G$, there exists $y \in S$ such that $xy \in E(G)$. Note that $E(G) \subseteq E(G+K_1)$. Thus, $xy \in E(G+K_1)$. Therefore, $S$ is a total dominating set of $G+K_1$.
	
	Now, let $x_1, x_2 \in S$. We want to show that $S$ is a 2-movable total dominating set of $G+K_1$. Since $S$ is a $\gamma_t$-set of $G$, at least one of the following conditions hold:
	\begin{enumerate}
		\item[(i)] $S \backslash \{x_1, x_2\}$ is a total dominating set of $G$ or
		\item[(ii)] for each $i$, there exists $y_i \in (V(G) \backslash S) \cap N(x_i)$ such that $(S \backslash \{x_1, x_2\}) \cup \{y_1, y_2\}$ is a total dominating set of $G$.
	\end{enumerate}	 
	
	Suppose $S \backslash \{x_1,x_2\}$ is a total dominating set of $G$. Let $x \in V(G+K_1)$. If $x \in V(G)$, there exists $y \in S \backslash \{x_1,x_2\}$ such that $xy \in E(G) \subseteq E(G+K_1)$. Suppose $x \in V(K_1)$. Let $x_3 \in S \backslash \{x_1,x_2\}$. Then $xx_3 \in E(G+K_1)$. Therefore, condition $(i)$ holds.
	
	Suppose condition $(ii)$ holds. Then for each $i$, there exists\\
	$y_i \in (V(G) \backslash S) \cap N_G(x_i)$ such that $S'=(S \backslash \{x_1, x_2\}) \cup \{y_1, y_2\}$ is a total dominating set of $G$. Since $V(G) \subseteq V(G+K_1)$ and $N_G(x_i) \subseteq N_{G+K_1} (x_i)$, then $y_i \in (V(G+K_1) \backslash S) \cap N_{G+K_1}(x_i)$. Let $x \in V(G+K_1)$. If $x \in V(G)$, then there exists $y \in S'$ such that $xy \in E(G) \subseteq E(G+K_1)$ since $S'$ is a total dominating set of $G$. Suppose $x \in V(K_1)$. Let $x_3 \in S'$. Then $xx_3 \in E(G+K_1)$. 
	
	Hence, $S$ is a 2-movable total dominating set of $G+K_1$. Therefore,\\
	$\gamma_{mt} (G+K_1) \leq |S|=\gamma_t (G)$.
	
	Now, suppose $\gamma_{mt} (G+K_1) \leq \gamma_t G$. Then there exists a minimum 2-movable total dominating set $T$ of $G+K_1$ such that $|T| < \gamma_t (G)$.
	
	Let $x_1, x_2 \in T$ such that $x_i \in V(K_1)$. Since $T$ is a 2-movable total dominating set of $G+K_1$, at least one of the following conditions holds:
	\begin{enumerate}
		\item[(i)] $T^* = T \backslash \{x_1, x_2\}$ is a total dominating set of $G+K_1$ or
		\item[(ii)] for each $i$, there exists $y_i \in (V(G+K_1) \backslash T) \cap N_{G+K_1}(x_i)$ such that\\
		$T' = (T \backslash \{x_1, x_2\}) \cup \{y_1, y_2\}$ is a total dominating set of $G+K_1$.
	\end{enumerate}	
	
	Suppose $(i)$ holds. Let $x \in V(G) \subseteq V(G+K_1)$. Then, there exists $y \in T^*$ such that $xy \in E(G+K_1)$. Since $x, y \in V(G)$, $xy \in E(G)$. Thus, $T^*$ is a total dominating set of $G$. But this means that $|T^*| < |T| < \gamma_t (G)$. This is a contradiction since $\gamma_t (G)$ is the minimum cardinality of a total dominating set in G.
	
	Suppose $(ii)$ holds. Let $x \in V(G) \subseteq V(G+K_1)$. Then, there exists $y \in T'$ such that $xy \in E(G+K_1)$. Since $x,y \in V(G)$, $xy \in E(G)$. Thus, $T'$ is a total dominating set of $G$. But this means that $|T'|=|T| < \gamma_t G$. This is a contradiction since $\gamma_t (G)$ is the minimum cardinality of a total dominating set in $G$. Therefore, $$\gamma_{mt}^2 (G+K_1) = \gamma_t G.$$
\end{proof}

\begin{lemma}\rm\label{Lem:2}
	For any corona graph $G \circ H$ and $a \in V(G)$, if $T$ is 2-movable total dominating set of $G \circ H$, $a \in V(G) \cap T$, and $T_a = T \cap V(H^a)$, then one of the following holds:
	\begin{enumerate}
		\item[(i)] $T_a \backslash \{a,u\}$ where $u \in T_a$ is a total dominating set of $H^a$, or
		\item[(ii)] there exists $x_a$ and $x_u$ such that $ax_a, ux_u \in E(H^a)$ and $(T_a \backslash \{a,u\}) \cup \{x_a, x_u\}$ is a total dominating set of $H^a$
	\end{enumerate}
\end{lemma}

\begin{proof}
	Suppose $a \in V(G)$ and $T$ is a 2-movable total dominating set of $G \circ H$. Let $u \in T_a$. Since $T$ is a 2-movable total dominating set of $G \circ H$, then one of the following conditions holds:
	\begin{enumerate}
		\item[(1)] $T \backslash \{a,u\}$ is a total dominating set of $G \circ H$, or
		\item[(2)] there exists $b$ and $v$ such that $ab, uv \in E(G \circ H)$ and $(T \backslash \{a,u\}) \cup \{b, v\}$ is a total dominating set of $G \circ H$
	\end{enumerate}
	
	Suppose $(1)$ holds. Let $x \in V(H^a) \subseteq V(G \circ H)$. Since $T \backslash \{a,u\}$ is a total dominating set of $G \circ H$, there exists $y \in T \backslash \{a,u\}$ such that $xy \in E(G \circ H)$. Note that $y \in V(H^a)$. It follows that $y \in T_a \backslash \{a,u\}$. Thus, $T_a \backslash \{a,u\}$ where $u \in T_a$, is a total dominating set of $H^a$.
	
	Suppose $(2)$ holds. Let $b \in V(H^a)$. Further, suppose $x_a=b$ and $x_u=v$. Let $x \in V(H^a) \subseteq V(G \circ H)$. Since $(T \backslash \{a,u\}) \cup \{b,v\}$ is a total dominating set of $G \circ H$, there exists $y \in (T \backslash \{a,u\}) \cup \{b,v\}$ such that $xy \in E(G \circ H)$. Note that $y \in V(H^a)$ since $x \in V(H^a)$ and $y \not= a$. Thus, $y \in (T_a \backslash \{a,u\}) \cup \{b,v\}$ since $[(T \cap V(H^a)) \backslash \{a,u\}] \cup \{b,v\} = (T_a \backslash \{a,u\}) \cup \{b,v\}$. Since $y \in V(H^a)$ and $b \notin V(a + H^a)$, $y \not= b$. It follows that $y \in (T_a \backslash \{a,u\}) \cup \{v\} \subseteq (T_a \backslash \{a,u\}) \cup \{x_a, v\}$. Moreover, $xy \in E(H^a)$. Hence, $(T_a \backslash \{a,u\}) \cup \{x_a,v\}$ is a total dominating set of $H^a$.
\end{proof}	

\begin{lemma}\rm\label{Lem:3}
	Let $T$ be a total dominating set of $G \circ H$ and $a \in V(G)$. If $a \notin T$, then $T \cap V(H^a)$ is a total dominating set of $H^a$.
\end{lemma}

\begin{proof}
	Suppose $T$ is a total dominating set of $G \circ H$ and $a \in V(G) \backslash T$. Let $x \in V(H^a)$. Since $T$ is a total dominating set of $G \circ H$, there exists $y \in T$ such that $xy \in E(G \circ H)$. Note that $y \in V(H^a)$ since $xy \in E(G \circ H)$ and $x \in V(H^a)$. Moreover, $y \in T \cap V(H^a)$. Thus, $T \cap V(H^a)$ is a total dominating set of $H^a$.
\end{proof}	

\begin{theorem}\label{Thm:16}\rm
	Let $G$ and $H$ be connected graphs such that $|V(G \circ H)| \geq 3$ and $\gamma_t (H) < |V(G)|$. Then
	\begin{center}
		$\gamma_{mt}^2 (G \circ H)=|V(G)| \gamma_t (H)$.
	\end{center}
\end{theorem}

\begin{proof}
	For each $x \in V(G)$, let $S_x$ be a $\gamma_t$-set of $H^x$. Further, let $S = \displaystyle \bigcup_{x \in V(G)} S_x$ and $a \in V(G \circ H)$. If $a \in V(G)$, then there exists $b \in S_a \subseteq S$ such that $ab \in E(G \circ H)$. Suppose $a \in V(H^y)$ for some $y \in V(G)$. Since $S_y$ is a $\gamma_t$-set of $H^y$, then there exists $b \in S_y \subseteq S$ such that $ab \in E(H^y) \subseteq E(G \circ H)$. Thus, $S$ is a total dominating set of $G \circ H$.
	
	Let $u,v \in S$. Then there exists $a,b \in V(G)$ such that $u \in S_a$ and $v \in S_b$. To show that $S$ is a 2-movable total dominating set of $G \circ H$, we consider the following cases:
	
	\noindent \textbf{Case 1:} $a =b$
	
	Since $\gamma_t(H) < |V(H)|$, there exists $w \in S_a$ such that $u \not= w$ and $v \not= w$. Let $S'=(S \backslash \{u,v\}) \cup \{a,w\}$. Since $S$ is a total dominating set of $G \circ H$ and $(S_a \backslash \{u,v\}) \cup \{a,w\}$ is a total dominating set of $H^a$, it follows that $S'$ is a total dominating set of $G \circ H$.
	
	\noindent \textbf{Case 2:} $a \not=b$
	
	Since $S$ is a total dominating set of $G \circ H$, $(S_a \backslash \{u\}) \cup \{a\}$ is a total dominating set of $H^a$ and $(S_b \backslash \{v\}) \cup \{b\}$ is a total dominating set of $H^b$, then $S' = (S \backslash \{u,v\}) \cup \{a,b\}$ is a total dominating set of $G \circ H$.
	
	Hence, $S$ is a 2-movable total dominating set of $G \circ H$. Therefore,
	\begin{align*}
		\gamma_{mt} (G \circ H) \leq |S| &= \displaystyle \sum_{x \in V(G)} |S_x|= \displaystyle \sum_{x \in V(G)} \gamma_t (H^x) \\
		&= \displaystyle \sum_{x \in V(G)} \gamma_t (H) = |V(G)| \cdot \gamma_t (H)
	\end{align*}	
	
	Suppose $\gamma_{mt} (G \circ H) < |V(G)| \gamma_{t}(H)$. Then there exists a $\gamma_{mt}^2$-set $T$ of $G \circ H$ such that $|T|< |V(G)| \gamma_{t}(H)$. For each $x \in V(G)$, let $T_x=T \cap V(x + H^x)$. Since $|T|<|V(G)| \gamma_{t}(H) $, there exists $a \in V(G)$ such that $|T_a| < \gamma_t (H) = \gamma_t (H^a)$. Suppose $a \in T$. Then $T_a = S_a \cup \{a\}$, that is, $|T_a| = |S_a| + 1$. By Lemma \ref{Lem:2}, one of the following conditions holds:
	\begin{enumerate}
		\item[(i)] $T_a \backslash \{a,u\}$ where $u \in T_a$ is a total dominating set of $H^a$, or
		\item[(ii)] there exists $x_a$ and $x_u$ such that $ax_a, ux_u \in E(H^a)$ and $(T_a \backslash \{a,u\}) \cup \{x_a, x_u\}$ is a total dominating set of $H^a$
	\end{enumerate}
	
	Suppose $(i)$ holds. Then $S_a \backslash \{a,u\}$ is a total dominating set of $H^a$. Now, $|S_a \backslash \{a,u\}| < |S_a| < |T_a| <\gamma_t (H^a)$. This is a contradiction since $\gamma_t (H^a)$ is the minimum cardinality of total dominating set of $H^a$. Suppose $(ii)$ holds. Then $(S_a \backslash \{a,u\}) \cup \{x_a, x_u\}$ is a total dominating set of $H^a$. Now, observe that $|(S_a \backslash \{a,u\}) \cup \{x_a, x_u\}| = |S_a| < |T_a| < \gamma_t (H^a)$. This is a contradiction since $\gamma_t (H^a)$ is the minimum cardinality of total dominating set of $H^a$.
	
	Suppose $a \in T$. Then $T_a = S_a$, that is, $|T_a| = |S_a|$. By Lemma \ref{Lem:3}, $S_a$ is a total dominating set of $H^a$. Now, $|S_a| = |T_a| < \gamma_t (H^a)$. This is a contradiction since $\gamma_t (H^a)$ is the minimum cardinality of total dominating set of $H^a$. Therefore, $$\gamma_{mt} (G \circ H) = |V(G)| \gamma_t(H).$$
\end{proof}

\end{document}